\documentclass[11pt]{article}
\usepackage[english]{babel}
\usepackage{amssymb,amsmath,amsthm}
\newtheorem{theorem}{Theorem}[section]
\newtheorem{lemma}[theorem]{Lemma}
\newtheorem{proposition}[theorem]{Proposition}

\newtheorem{question}[theorem]{Question}
\newtheorem{example}[theorem]{Example}

{\theoremstyle{definition}\newtheorem{definition}[theorem]{Definition}}
{\theoremstyle{definition}}

\numberwithin{equation}{section}

\def\C{{\mathbb C}}
\def\A{{\mathbb A}}

\def\N{{\mathbb N}}
\def\Z{{\mathbb Z}}
\def\R{{\mathbb R}}

\def\epsilon{\varepsilon}
\def\phi{\varphi}
\def\kappa{\varkappa}
\def\leq{\leqslant}
\def\geq{\geqslant}
\def\ker{{\tt ker}\,}

\title{Nice Banach Modules and Invariant Subspaces}

\author{Stanislav Shkarin}

\date{}

\begin{document}

\maketitle

\begin{abstract} Let $\A$ be a semisimple unital commutative Banach algebra. We say that a Banach $\A$-module $M$ is {\it nice} if every proper closed submodule of $M$ is contained in a closed submodule of $M$ of codimension $1$. We provide examples of nice and non-nice modules.
\end{abstract}
\small \noindent{\bf MSC:} \ \ 46J10, 47A15

\noindent{\bf Keywords:} \ \ Banach algebras, Banach modules, Invariant subspaces
\normalsize

\section{Introduction \label{s1}}

In this article, all vector spaces are assumed to be over the field $\C$ of complex numbers. As usual, $\R$ is the field of real numbers, $\N$ is the set of all positive integers, $\Z$ is the set of integers and $\Z_+$ is the set of non-negative integers. For a Banach space $X$, $L(X)$ stands for the  algebra of bounded linear operators on $X$, while $X^*$ is the space of continuous linear functionals on $X$. For $T\in L(X)$, its dual is denoted $T^*$: $T^*\in L(X^*)$, $T'f(x)=f(Tx)$ for every $f\in X^*$ and every $x\in X$.

Throughout this article $\A$ stands for a unital commutative semisimple Banach algebra. It is well-known and is a straightforward application of the Gelfand theory \cite{helem,da} that for an ideal $J$ in $\A$,
$$
J=\A\iff \text{$J$ is dense in $\A$} \iff \text{$\kappa\bigr|_J\neq 0$ for every $\kappa\in\Omega(\A)$}, 
$$
where $\Omega(\A)$ is the spectrum of $\A$, that is, $\Omega(\A)$ is the set of all (automatically continuous) non-zero  algebra homomorphisms from $\A$ to $\C$ (endowed with the $*$-weak topology). Equivalently, every proper ideal in $\A$ is contained in a closed ideal of codimension 1. 

Let $\Omega^+(\A)$ be the set of all algebra homomorphisms from $\A$ to $\C$. That is, $\Omega^+(\A)$ is $\Omega(\A)$ together with the identically zero map from $\A$ to $\C$. The main purpose of this paper is to draw attention to possible extensions of the above fact to Banach $\A$-modules. Clearly, each $\kappa\in\Omega^+(\A)$ gives rise to the 1-dimensional $\A$-module $\C_\kappa$ being $\C$ with the $\A$-module structure given by the multiplication $a\lambda=\kappa(a)\lambda$ for every $a\in\A$ and $\lambda\in\C$. It is also rather obvious that we have just listed all the $1$-dimensional $\A$-modules up to an isomorphism. 

\begin{definition}\label{modchar} Let $M$ be a Banach $\A$-module. A {\it character} on $M$ is a non-zero $\phi\in M^*$ such that there exists $\kappa\in\Omega^+(\A)$ making $\phi$ into an
$\A$-module morphism from $M$ to $\C_\kappa$.
\end{definition}

Obviously, the kernel of a character on a Banach $\A$-module $M$ is a closed $\A$-submodule of $M$.

\begin{definition}\label{nice} Let $M$ be a Banach $\A$-module. We say that $M$ is {\it nice} if for every proper closed submodule of $M$ is contained in a closed submodule of codimension 1. Equivalently, $M$ is nice if and only if for every proper closed submodule $N$ of $M$, there is a character $\phi$ on $M$ such that $\phi$ vanishes on $N$. 
\end{definition}

The general question we would like to raise is: 

\begin{question}\label{mq} Characterize nice Banach $\A$-modules.
\end{question}

The remark we started with ensures that $\A$ is nice as an $\A$-module. In this paper we just present examples of nice and non-nice modules. Before even formulating the results, I would like to put forth my personal motivation for even looking at this question. Assume for a minute that $\A$ is a subalgebra of $L(X)$ for some Banach space $X$. We allow the norm topology of $\A$ to be stronger (not necessarily strictly) than the topology defined by the norm inherited from $L(X)$. The multiplication $(A,x)\mapsto Ax$ defines a Banach $\A$-module structure on $X$. What are the characters on $X$? Why, one easily sees that they are exactly the common eigenvectors of $A^*$ for $A\in\A$. What are the $\A$-submodules of $X$? They are exactly the invariant subspaces for the action of $\A$ on $X$. Thus the $\A$-module $X$ is nice exactly when every non-trivial closed $\A$-invariant subspace of $X$ is contained in a closed $\A$-invariant hyperplane. Thus $X$ being a nice $\A$-module translates into a strong and important property of the lattice of $\A$-invariant subspaces. Note that under relatively mild extra assumptions on $\A$, the nicety of $X$ results in every closed $\A$-invariant subspace being the intersection of a collection of characters on $X$ thus providing a complete description of the lattice of $\A$. A byproduct of this observation is the following easy example of a non-nice module.

\begin{example}\label{ex1} Let $\Omega$ be a non-empty compact subset of $\C$ with no isolated points and $\mu$ be a finite $\sigma$-additive purely non-atomic Borel measure on $\C$, whose support is exactly $\Omega$. The pointwise multiplication equips $L^2(\Omega,\mu)$ with the structure of a Banach $C(\Omega)$-module. This module is non-nice.
\end{example}

\begin{proof} The $C(\Omega)$-module $L^2(\Omega,\mu)$ does have plenty of closed submodules. For instance, every Borel subset $A$ of $\Omega$ satisfying $\mu(A)\neq 0$ and $\mu(\Omega\setminus A)\neq 0$ generates a closed non-trivial submodule $M_A=\{f\in L^2(\Omega,\mu):f\ \text{vanishes outside}\ A\}$. On the other hand, we can always pick $f\in C(\Omega)$ satisfying $\mu (f^{-1}(\lambda))=0$ for every $\lambda\in\C$. In this case the dual of the multiplication by $f$ operator on $L^2(\Omega,\mu)$ has empty point spectrum. Due to the above remark, our module possesses no characters at all (while possessing non-trivial closed submodules) and therefore can not possibly be nice.
\end{proof}

In the positive direction we have the following two rather easy statements. 

\begin{proposition}\label{fge} The finitely generated free $\A$-module $\A^n$ is nice. 
\end{proposition}

\begin{proposition}\label{cfu} Let $\Omega$ be a Hausdorff compact topological space and $X$ be a Banach space. Then the $C(\Omega)$-module $C(\Omega,X)$ is nice, where $C(\Omega,X)$ carries the natural norm $\|f\|=\sup\{\|f(\omega)\|_X:\omega\in\Omega\}$ and the module structure is given by the pointwise multiplication. 
\end{proposition}

Note that Example~\ref{ex1} is rather cheatish since the non-nicety comes from the lack of characters. A really interesting situation is when a non-nice module possesses a separating set of characters. The following result says that this is quite possible. Recall that the Sobolev space $W^{1,2}[0,1]$ consists of the functions $f:[0,1]\to\C$ absolutely continuous on any bounded subinterval of $I$ and such that $f'\in L_2[0,1]$. The space $W^{1,2}[0,1]$ with the inner product
$$
\langle f,g\rangle_{1,2}=\int_0^1 (f(t)\overline{g(t)}+
f'(t)\overline{g'(t)})\,dt
$$
is a separable Hilbert space. We denote $\|f\|_{1,2}=\sqrt{\langle
f,f\rangle_{1,2}}$. Apart from being a Hilbert space, $W^{1,2}[0,1]$ is also a Banach algebra with respect to the pointwise multiplication (if one strives for the submultiplicativity of the norm together with the identity $\|1\|=1$, he or she has to pass to an equivalent norm).

We say that a function $f$ defined on $[0,1]$ and taking values in
a Banach space $X$ is {\it absolutely continuous} if there exists
an (automatically unique up to a Lebesgue-null set) Borel measurable function $g:[0,1]\to X$
such that
$$
\int_0^1\|g(t)\|\,dt<+\infty\  \ \text{and}\ \
\int_0^xg(t)\,dt=f(x)\ \ \text{for each}\ \ x\in[0,1],
$$
where the second integral is considered in the Bochner sense. We
denote the function $g$ as $f'$. If $H$ is a Hilbert space. The
symbol $W^{1,2}([0,1],H)$ stands for the space of absolutely
continuous functions $f:[0,1]\to H$ such that
$$
\int_0^1\|f'(t)\|^2\,dt<+\infty.
$$
The space $W^{1,2}([0,1],H)$ with the inner product
$$
\langle f,g\rangle=\int_0^1 (\langle f(t),g(t)\rangle_H+
\langle f'(t),g'(t)\rangle_H)\,dt
$$
is a Hilbert space and is separable if $H$ is separable. In any
case if $\{e_\alpha\}_{\alpha\in A}$  is an orthonormal basis of
$H$, then the space $W^{1,2}([0,1],H)$ is naturally identified with
the Hilbert direct sum of $|A|$ copies of $W^{1,2}[0,1]$:
$f\mapsto\{f_\alpha\}_{\alpha\in A}$, where $f_\alpha(t)=\langle
f(t),e_\alpha\rangle_H$. It is also clear that $W^{1,2}([0,1],H)$ is naturally isomorphic 
to the Hilbert space tensor product of $W^{1,2}[0,1]$ and $H$. Clearly, $W^{1,2}([0,1],H)$ is a Banach $W^{1,2}[0,1]$-module. This module possesses a lot of characters. Indeed, if $t\in[0,1]$ and $x\in H$, then the functional $f\mapsto \langle f(t),x\rangle_H$ is a character on $W^{1,2}([0,1],H)$. Moreover, these characters do separate points of $W^{1,2}([0,1],H)$.

\begin{theorem}\label{ma} Let $H$ be a Hilbert space. Then the $W^{1,2}[0,1]$-module $W^{1,2}([0,1],H)$ is nice if and only if $H$ is finite dimensional.
\end{theorem} 

\section{Proof of Proposition~\ref{cfu}} 

It is easy to see that a character on $C(\Omega,X)$ is exactly a functional of the form 
\begin{equation}\label{charc} 
\kappa_{\omega,\phi}(f)=\phi(f(\omega)),\ \ \text{where $\omega\in\Omega$ and $\phi\in X^*\setminus\{0\}$}.
\end{equation}

The following lemma describes all closed submodules of $C(\Omega,X)$.

\begin{lemma}\label{clmo} Let $M$ be a $C(\Omega)$-submodule of $C(\Omega,X)$ and for each $\omega\in\Omega$ let $M_\omega=\{f(\omega):f\in M\}$. Then the closure $\overline{M}$ of $M$ in $C(\Omega,X)$ satisfies 
\begin{equation}\label{cloM}
\overline{M}=\widetilde{M},\ \ \text{where}\ \widetilde{M}=\{f\in C(\Omega,X):f(\omega)\in \overline{M}_\omega\ \ \text{for each $\omega\in\Omega$}\},
\end{equation}
with $\overline{M}_\omega$ being the closure in $X$ of $M_\omega$. 
\end{lemma}

\begin{proof} Since $M\subseteq\widetilde{M}$ and $\widetilde{M}$ is closed, we have $\overline{M}\subseteq\widetilde{M}$. Let $f\in \widetilde{M}$ and $\epsilon>0$. The desired equality will be verified if we show that there is $g\in M$ such that $\|f-g\|<\epsilon$. Indeed, in this case $\widetilde{M}\subseteq\overline{M}$ and therefore $\overline{M}=\widetilde{M}$. 

Take $\omega\in\Omega$. Since $M_\omega$ is dense in $\overline{M}_\omega$, there is $g_\omega\in M$ such that $\|f(\omega)-g_\omega(\omega)\|_X<\epsilon$. Then $V_\omega=\{s\in\Omega:\|f(s)-g_\omega(s)\|_X<\epsilon\}$ is an open subset of $\Omega$ containing $\omega$. Thus $\{V_\omega\}_{\omega\in\Omega}$ is an open covering of $\Omega$. Since for every open covering of a Hausdorff compact topological space, there is a finite partition of unity consisting of continuous functions and subordinate to the covering \cite{engel}, there are $\omega_1,\dots,\omega_n\in\Omega$ and $\rho_1,\dots,\rho_n\in C(\Omega)$ such that 
\begin{equation}\label{paun}
\begin{array}{l} 0\leq \rho_j(s)\leq 1\ \ \text{for every $1\leq j\leq n$ and $s\in\Omega$;}\\
\rho_j(s)=0\ \ \text{whenever $1\leq j\leq n$ and $s\in\Omega\setminus V_{\omega_j}$;}\\
\rho_1(s)+{\dots}+\rho_n(s)=1\ \ \text{for each $s\in\Omega$.}
\end{array}
\end{equation}
Now we set $g=\rho_1g_{\omega_1}+{\dots}+\rho_ng_{\omega_n}$. Since $M$ is a $C(\Omega)$-module and $g_\omega\in M$, we have $g\in M$. Using (\ref{paun}) together with the inequality $\|f(s)-g_{\omega_j}(s)\|_X<\epsilon$ for $s\in V_{\omega_j}$, we easily see that $\|f(s)-g(s)\|_X<\epsilon$ for each $s\in\Omega$. Hence $g\in M$ and $\|f-g\|<\epsilon$, which completes the proof.
\end{proof}

{\bf We are ready to prove Proposition~\ref{cfu}.} \ Let $M$ be a closed submodule of $C(\Omega,X)$ such that none of the characters on $C(\Omega,X)$ vanishes on $M$. According to (\ref{charc}), the latter means that every $M_\omega=\{f(\omega):f\in M\}$ is dense in $X$ and therefore $\overline{M}_\omega=X$ for each $\omega\in\Omega$. Since $M$ is closed, Lemma~\ref{clmo} says that $M=C(\Omega,X)$. The proof is complete.

\section{Proof of Propositions~\ref{fge}} 

We start with the following easy observation. Let $\kappa\in\Omega(\A)$. Then the $\A$-module morphisms $\psi:\A^n\to\C_\kappa$ are all given by 
$$
\phi_c(a_1,\dots,a_n)=\sum_{j=1}^n c_j\kappa(a_j),\ \ \text{where $c\in \C^n$.}
$$

We shall prove a statement slightly stronger than Proposition~\ref{fge}. 

\begin{proposition}\label{fge1} Let $n\in\N$ and $M$ be an $\A$-submodule of the free $\A$-module $\A^n$. Assume also that none of the characters on $\A^n$ vanishes on $M$. Then $M=\A^n$.
\end{proposition}

\begin{proof}
We use induction with respect to $n$. The case $n=1$ is trivial (see the remark at the very start of the article). Assume now that $n\geq 2$ and that the conclusion of Proposition~\ref{fge} holds for every smaller $n$. We interpret $\A^n$ as $\A^n=\A\times \A^{n-1}$. The induction hypothesis easily implies that that the projection of $M$ onto $\A^{n-1}$ is onto. Let $J\subseteq \A$ be defined by $M\cap (\A\times\{0\})=J\times\{0\}$. Then $J$ is an ideal in $A$. If $J=\A$, we can factor out the first component in the product $\A\times \A^{n-1}=\A^n$ and then use the induction hypothesis to conclude that $M=\A^n$. Thus it remains to consider the case $J\neq\A$. Then there is $\kappa\in\Omega(\A)$ such that $J\subseteq\ker\kappa$. Using the definition of $J$, and the facts that $M$ is an $\A$-module, $M$ projects onto the entire $\A^{n-1}$ and $\kappa$ vanishes on $J$, we can define 
$\psi:\A^{n-1}\to \C$ by the rule $\psi(b)=\kappa(a)$ if $(a,b)\in M\subseteq \A\times \A^{n-1}$. It is easy to see that $\psi$ is a well-defined continuous linear functional and that $\psi:\A^{n-1}\to\C_\kappa$ is an $\A$-module morphism. According to the above display there are $c_1,\dots,c_{n-1}\in\C$ such that $\psi(a_1,\dots,a_{n-1})=\sum\limits_{j=1}^{n-1}c_j\kappa(a_j)$ for every $a_1,\dots,a_{n-1}\in\A$. By definition of $\psi$, we now see that $\phi:\A^n\to\C$ vanishes on $M$, where $\phi$ is defined by the formula $\phi(a_1,\dots,a_{n})=\sum\limits_{j=1}^{n}c_j\kappa(a_j)$ with $c_n=-1$. By the above display, $\phi:\A^n\to\C_\kappa$ is an $\A$-module morphism. Since $c_n\neq 0$, $\phi\neq 0$ and therefore $\phi$ is a character on $\A^n$. We have produced a character on $\A^n$ vanishing on $M$, which contradicts the assumptions. Thus the case $J\neq \A$ does not occur, which completes the proof.
\end{proof}

\section{Proof of Theorem~\ref{ma}} 

In this section, for a function $f$ on an interval $I$ of the real line $\|f\|_2$ will always denote the $L^2$-norm of $f$ (with respect to the Lebesgue measure), while $\|f\|_\infty$ always stands for the $L^\infty$-norm of $f$. 

\begin{lemma}\label{5.7}Let $-\infty<\alpha<\beta<+\infty$, $a,b\in\C$ and
$\epsilon>0$. Then there exists $f\in C^1[\alpha,\beta]$ such that
$f(\alpha)=f(\beta)=0$, $f'(\alpha)=a$, $f'(\beta)=b$ and
$\|f\|_\infty<\epsilon$. \end{lemma}

\begin{proof} Let $\phi\in C^1[0,\infty)$ be a monotonically
non-increasing function such that $\phi(0)=1$, $\phi'(0)=0$ and
$\phi(x)=0$ for $x\geq1$. For any
$\delta\in(0,\frac{\beta-\alpha}2)$ let
$$
f_\delta(x)=\left\{\begin{array}{ll}0&\text{if
$x\in(\alpha+\delta,\beta-\delta)$},\\
a(x-\alpha)\phi((x-\alpha)/\delta)&\text{if
$x\in[\alpha,\alpha+\delta)$},\\
b(x-\beta)\phi((\beta-x)/\delta)&\text{if
$x\in(\beta-\delta,\beta]$}.\end{array}\right.
$$
Obviously, $f_\delta\in C^1[\alpha,\beta]$,
$f_\delta(\alpha)=f_\delta(\beta)=0$, $f_\delta'(\alpha)=a$,
$f_\delta'(\beta)=b$ and $\|f\|_\infty\leq\delta\max\{|a|,|b|\}$.
Hence the function $f=f_\delta$ for
$\delta<\epsilon/\max\{|a|,|b|\}$ satisfies all desired
conditions. \end{proof}

\begin{lemma}\label{5.8}Let $K\subset[0,1]$ be a nowhere dense compact set,
$a\in C(K)$, $f\in C[0,1]$ and $\epsilon>0$. Then there exists
$g\in C^1[0,1]$ such that $g'\bigr|_{K}=a$ and
$\|g-f\|_\infty<\epsilon$. \end{lemma}

\begin{proof} Since $C^1[0,1]$ is dense in the Banach space $C[0,1]$, we
can, without loss of generality, assume that $f\in C^1[0,1]$.
Since any continuous function on $K$ admits a continuous extension
to $[0,1]$ (one can apply, for instance, the Tietze theorem
\cite{engel}), there exists $h\in C[0,1]$ such that
$h(x)=a(x)-f'(x)$ for any $x\in K$. Let $\delta>0$. Since $K$ is
nowhere dense, there exist
$$
0=\alpha_1<\beta_1<\alpha_2<\beta_2<\dots<\alpha_n<\beta_n=1
$$
such that $\beta_j-\alpha_j<\epsilon$ for any $j=1,\dots,n$ and
$K\subset\bigcup\limits_{j=1}^n I_j$, where
$I_j=[\alpha_j,\beta_j]$. Let
$$
a_j=\int_{\alpha_j}^{\beta_j} h(t)\,dt\ \ \text{for $1\leq j\leq n-1$}.
$$
By Lemma~\ref{5.7}, for $1\leq j\leq n-1$, there is
$\phi_j\in C^1[\beta_j,\alpha_{j+1}]$ such that
$\phi_j(\beta_j)=\phi_j(\alpha_{j+1})=0$,
$\phi_j'(\beta_j)=h(\beta_j)+\frac{a_j}{\alpha_{j+1}-\beta_j}$,
$\phi_j'(\alpha_{j+1})=h(\alpha_{j+1})+\frac{a_j}{\alpha_{j+1}-\beta_j}$
and $\|\phi_j\|_\infty<\delta$. Consider the function
$$
\psi(x)=\left\{\begin{array}{ll}\int_{\alpha_j}^xh(t)\,dt&\text{if
$x\in[\alpha_j,\beta_j]$, $1\leq j\leq n$},\\
\phi_j(x)+\frac{a_j(x-\alpha_{j+1})}{\beta_j-\alpha_{j+1}}&\text{if
$x\in(\beta_j,\alpha_{j+1})$, $1\leq j\leq n-1$}.\end{array}\right.
$$
The values of $\phi_j'$ at $\beta_j$ and $\alpha_{j+1}$
were chosen in such a way that $\psi\in C^1[0,1]$. Moreover, $\psi'\bigr|_{I_j}=h$ for
$1\leq j\leq n$. Hence,
$(\psi+f)'\bigr|_{K}=a$. Let us estimate $\|\psi\|_\infty$. If $1\leq j\leq n-1$ and 
$x\in[\beta_j,\alpha_{j+1}]$, then 
$|\psi(x)|\leq\delta+|a_j|\leq
\delta+|\beta_j-\alpha_j|\|h\|_\infty\leq \delta(1+\|h\|_\infty)$.
If $1\leq j\leq n$ and $x\in[\alpha_j,\beta_j]$, then $|\psi(x)|\leq
|\beta_j-\alpha_j|\|h\|_\infty\leq\delta\|h\|_\infty$. Hence
$\|\psi\|_\infty\leq \delta(1+\|h\|_\infty)$. Choose 
$\delta<\epsilon/(1+\|h\|_\infty)$ and denote $g=\psi+f$. Then
$g'\bigr|_{K}=a$ and $\|g-f\|_\infty=\|\psi\|_\infty<\epsilon$.
\end{proof}

\begin{lemma}\label{5.9}Let $K\subset[0,1]$ be a nowhere dense compact set and
$\epsilon>0$. Then there exists $f\in C(K)$ such that
$$
\int_K f(t)\,dt=0\ \ \text{and}\ \
\|\chi+g\|_2\leq\epsilon,\ \ \text{where}\ \ g(x)=\int_{K\cap[x,1]}f(t)\,dt
$$
and $\chi$ is the indicator function of $K$ $(\chi(x)=1$ if $x\in K$ and $\chi(x)=0$ if $x\in[0,1]\setminus K)$. 
\end{lemma}

\begin{proof} If the Lebesgue measure $\mu(K)$ of $K$ is zero, the
statement is trivially true since the function $f\equiv 0$
satisfies the desired conditions for any $\epsilon>0$. Thus, we
can assume that $\mu(K)>0$. Let $n\in\N$. Since $K$ is nowhere
dense and has positive Lebesgue measure, we can choose $n\in\N$ and 
$\alpha_k,\beta_k,a_k,b_k,u_k,v_k\in[0,1]\setminus K$ for $1\leq k\leq n$ in such
a way that
\begin{align}
&\alpha_k<\beta_k<a_k<b_k<u_k<v_k\ \ \text{for}\ \ 1\leq k\leq n\ \ 
\text{and}\ \ v_{k-1}<\alpha_k\ \ \text{for}\ \ 2\leq k\leq n,
\notag
\\
&\textstyle0<\mu(K\cap[\alpha_k,\beta_k])<\frac{\epsilon^2}{16n}\ \ \text{and}\ \
0<\mu(K\cap[u_k,v_k])<\frac{\epsilon^2}{16n}\ \ \text{for}\ \ 1\leq k\leq
n, \label{mu1}
\\
&\textstyle\mu\Bigl(\Bigl(\bigcup\limits_{k=1}^n[\alpha_k,v_k]\Bigr)\setminus
K\Bigr)<\frac{\epsilon^2}{8}.\label{mu2}
\end{align}
Consider the function $f:K\to\R$ defined by the formula
$$
f(x)=\left\{\begin{array}{ll}\frac{1}{\mu(K\cap[\alpha_k,\beta_k])}&\text{if}\
x\in K\cap[\alpha_k,\beta_k],\ 1\leq k\leq n;\\
\frac{-1}{\mu(K\cap[u_k,v_k])}&\text{if}\ x\in K\cap[u_k,v_k],\
1\leq k\leq n;\\
0&\text{otherwise}.
\end{array}\right.
$$
Obviously $f\in C(K)$ and
$$
\int_K
f(t)\,dt=\sum_{k=1}^n\biggl(\int_{K\cap[\alpha_k,\beta_k]}
f(t)\,dt-\int_{K\cap[u_k,v_k]}
f(t)\,dt\biggr)=\sum_{k=1}^n(1-1)=0.
$$
Let $g:[0,1]\to\R$ be defined by  
$$
g(x)=\int_{K\cap[x,1]}f(t)\,dt.
$$ 
From the definition
of $f$ it follows that $|g(x)|\leq 1$ for any $x\in[0,1]$,
$g(x)=-1$ if $\displaystyle x\in\bigcup_{k=1}^n[\beta_k,u_k]$ and
$\chi(x)=g(x)=0$ if
$x\in[0,1]\setminus\bigcup\limits_{k=1}^n[\alpha_k,v_k]$. Hence the set
$\Omega=\{x\in[0,1]:g(x)+\chi(x)\neq 0\}$ is contained in the
union
$$
\Omega_1=\left(\left(\bigcup_{k=1}^n[\alpha_k,v_k]\right)\setminus
K\right)\cup\left(\bigcup_{k=1}^n([\alpha_k,\beta_k]\cap K)\right)
\cup\left(\bigcup_{k=1}^n([u_k,v_k]\cap K)\right).
$$
Therefore
$$
\|g+\chi\|_2^2=\int\limits_0^1(g(x)+\chi(x))^2\leq
4\mu(\Omega)\leq 4\mu(\Omega_1).
$$
Using (\ref{mu1}) and (\ref{mu2}), we see that
$\mu(\Omega_1)\leq\epsilon^2/4$. Hence $\|g+\chi\|_2\leq \epsilon$.
\end{proof}

\begin{lemma}\label{5.10}Let $\{e_n\}_{n\in\Z_+}$ be an orthonormal basis in a
separable Hilbert space $H$ and scalar sequences
$\{\gamma_n\}_{n\in\N}$ and $\{\delta_n\}_{n\in\N}$ be such that
\begin{equation}
\sum_{n=1}^\infty (|\gamma_n|^2+|\delta_n|^2)<\infty.\label{garho}
\end{equation}
Let also $f_0=e_0+\sum\limits_{n=1}^\infty\gamma_ne_n$ and
$f_n=e_n-\delta_n e_0$ for $n\in\N$. Then the linear span of
$\{f_n:n\in\Z_+\}$ is dense in $H$ if and only if
\begin{equation}
\sum_{n=1}^\infty\gamma_n\delta_n\neq-1. \label{span}
\end{equation}\end{lemma}

\begin{proof} Condition (\ref{garho}) implies that the linear operator
$T:H\to H$ such that $Te_0=\sum\limits_{n=1}^\infty\gamma_ne_n$
and $Te_n=-\delta_ne_0$ for $n\in\N$ is bounded. Since the range
of $T$ is at most two-dimensional, $T$ is compact. By the 
Fredholm theorem \cite{rudin1}, the operator $S=I+T$ has dense
range if and only if $S$ is injective. Since $Se_n=f_n$ for 
$n\in\Z_+$, the linear span of $\{f_n\}_{n\in\Z_+}$
is dense in $H$ if and only if the operator $S$ injective.

The equation $Sx=0$, $x\in H$ can be rewritten as
$$
\langle
x,e_0\rangle\left(1+\sum_{n=1}^\infty\gamma_n\delta_n\right)=0\ \
\text{and}\ \ \langle x,e_n\rangle=\gamma_n\langle x,e_0\rangle\ \
\text{for any}\  \ n\in\N.
$$
If $\sum\limits_{n=1}^\infty\gamma_n\delta_n\neq-1$, the first
equation implies $\langle x,e_0\rangle=0$ and the rest yield $\langle x,e_n\rangle=0$ for each $n\in\N$. Thus in this case  $x=0$. That is, $S$ is injective and therefore the linear
span of $\{f_n:n\in\Z_+\}$ is dense in $H$. If
$\sum\limits_{n=1}^\infty\gamma_n\delta_n=-1$, the system of the
equations in the above display has the non-zero solution
$x=x_0+\sum\limits_{n=1}^\infty \gamma_ne_n\in H$. Hence $S$ is
not injective and therefore the linear span of
$\{f_n:n\in\Z_+\}$ is non-dense. \end{proof}

{\bf We are ready to prove Theorem~\ref{ma}.} \ First, if $n\in\N$ and $H$ is $n$-dimensional, then $W^{1,2}([0,1],H)$ is isomorphic to the free $W^{1,2}[0,1]$-module with $n$ generators and the nicety of $W^{1,2}([0,1],H)$ follows from Proposition~\ref{fge}. It is easy to see that a direct (module) summand of a nice module is nice. Thus the proof of Theorem~\ref{ma} will be complete if we verify that $W^{1,2}([0,1],\ell_2)$ is non-nice. In order to do this, we have to construct a proper closed $W^{1,2}[0,1]$-submodule $M$ of $W^{1,2}([0,1],\ell_2)$ such that none of the characters on $W^{1,2}([0,1],\ell_2)$ vanishes on $M$. Now we shall do just that. 

Pick a nowhere dense compact set $K\subset[0,1]$ of positive Lebesgue measure and
let $\chi$ be the indicator function of $K$. By Lemma~\ref{5.9}, there exists $A_n\in C(K)$ such that for any $n\in\N$,
\begin{align}
&\int\limits_K A_n(x)\,dx=0,\label{an1}
\\
&\|B_n+\chi\|_2<2^{-n},\ \ \text{where}\  \
B_n(x)=\int\limits_{K\cap[x,1]}A(t)\,dt.\label{an2}
\end{align}
We also set $A_0=0$, $B_0=0$ and $S_0=1$. By Lemma~\ref{5.8}, there
exist $S_n\in C^1[0,1]$ such that
\begin{equation}
S_n'\bigr|_{K}=A_n\  \ \text{and}\  \ \|S_n-1\|_\infty<2^{-n}\ \ \text{for each $n\in\N$}.
\label{sn}
\end{equation}
Denote $\rho_n=n^2(S_n-S_{n-1})$ for $n\in\N$. Then $\rho_n\in
C^1[0,1]$ and according to (\ref{sn}),
\begin{equation}
\|\rho_n\|_\infty\leq
n^2(\|S_n-1\|_\infty+\|S_{n-1}-1\|_\infty)\leq
n^2(2^{1-n}+2^{-n})=3n^22^{-n}\ \ \text{for each $n\in\N$}. \label{rhon1}
\end{equation}
Let also $\{e_n\}_{n\in\Z_+}$ be the standard orthonormal basis in
$\ell_2$. Consider the functions $f^{[n]}\in W^1_2([0,1],\ell_2)$
defined by the formulas
$$
f^{[0]}(x)=e_0+\sum_{n=1}^\infty n^{-2}e_n\ \ \text{and}\ \ 
f^{[n]}(x)=e_n-\rho_n(x)e_0\ \ \text{for $n\in\N$}.
$$
Let now $M$ be the closed $W^{1,2}[0,1]$-submodule of $W^{1,2}([0,1],\ell_2)$ generated by the set $\{f^{[n]}:n\in\Z_+\}$. Equivalently, $M$ is the closed linear span in $W^{1,2}([0,1],\ell_2)$ of
the set $\{\phi f^{[n]}:n\in\Z_+,\ \phi\in W^{1,2}[0,1]\}$. 

It is easy to see that every character on $W^{1,2}([0,1],\ell_2)$ has the shape 
$$
\phi_{t,y}(f)=\langle f(t),y\rangle_H,\ \ \text{where $t\in[0,1]$ and $y\in\ell_2\setminus\{0\}$.}
$$
Thus in order for every character on $W^{1,2}([0,1],\ell_2)$ not to vanish on $M$ it is necessary and sufficient for $M_t=\{f(t):f\in M\}$ to be dense in $\ell_2$ for every $t\in[0,1]$. Let $t\in[0,1]$. By definition of $\rho_n$ and (\ref{sn}), we have
\begin{equation}
\sum_{n=1}^\infty
n^{-2}\rho_n(t)=\lim_{m\to\infty}\sum_{n=1}^m(S_n(t)-S_{n-1}(t))=
\lim_{m\to\infty}(S_m(t)-S_0(t))=0\neq -1.\label{rhon2}
\end{equation}
By Lemma~\ref{5.10} with $\gamma_n=n^{-2}$ and $\delta_n=\rho_n(t)$, the linear span of 
$\{f^{[n]}(t)\}_{n\in\Z_+}$ is dense in $\ell_2$. Since $f^{[n]}\in M$, $M_t$ is dense in $\ell_2$. Thus none of the characters  on $W^{1,2}([0,1],\ell_2)$ vanishes on $M$. It remains to verify that $M\neq W^{1,2}([0,1],\ell_2)$. Consider $g_n\in W_2^1[0,1]^*$ for $n\in\Z_+$,  defined by the formula
$$
g_n(\phi)=\int_K(\rho_n\phi)'(x)\,dx,\ \ \text{where $\rho_0$ is assumed to be identically 1.}
$$
We start with estimating the norms of the functionals $g_n$. Clearly, 
\begin{equation}
g_n(\phi)=\int_K\rho_n(x)\phi'(x)\,dx+
\int_K\rho'_n(x)\phi(x)\,dx\ \ \text{for any $\phi\in W^{1,2}[0,1]$}. \label{gn1}
\end{equation}
Since $\rho'_n(x)=n^2(S_n'(x)-S_{n-1}'(x))=n^2(A_n(x)-A_{n-1}(x))$ for $x\in K$, we have 
$$
\int_K\rho'_n(x)\phi(x)\,dx=
n^2\int_K(A_n(x)-A_{n-1}(x))\phi(x)\,dx
=n^2\int_0^1(B'_{n-1}(x)-B'_{n}(x))\phi(x)\,dx.
$$
By (\ref{an1}) and (\ref{an2}), 
$$
B_n(0)=B_n(1)=0\ \ \text{for $n\in\Z_+$.}
$$
Integrating by parts and using the above display, we obtain
$$
\int_K\rho'_n(x)\phi(x)\,dx=n^2\int_0^1(B'_{n-1}(x)-B'_{n}(x))\phi(x)\,dx=
n^2\int_0^1(B_{n}(x)-B_{n-1}(x))\phi'(x)\,dx.
$$
This formula together with (\ref{gn1}) yields 
$$
|g_n(\phi)|\leq \|\phi'\|_2(\|\rho_n\|_2+n^2\|B_n-B_{n-1}\|_2)\ \ \text{for $n\in\N$}.
$$
Since $\|\rho_n\|_2\leq3n^22^{-n}$ and
$\|B_n-B_{n-1}\|_2\leq \|B_n+\chi\|_2+\|B_{n-1}+\chi\|_2\leq
2^{1-n}+2^{-n}=3\cdot2^{-n}$, 
we have $|g_n(\phi)|\leq 6n^22^{-n}\|\phi\|_{1,2}$.
Hence $\|g_n\|\leq 6n^22^{-n}$ for each $n\in\N$. Therefore
$\sum\limits_{n=0}^\infty\|g_n\|^2<\infty$. Thus the formula
$$
g(h)=\sum_{n=0}^\infty g_n(h_n)
$$
defines a continuous linear functional on $W^{1,2}([0,1],\ell_2)$, where, as usual, $h_n(t)=\langle h(t),e_n\rangle$. Since $g_0\neq0$, we have $g\neq 0$. In order to show that $M\neq W^{1,2}([0,1],\ell_2)$, it suffices to verify that 
$g(h)=0$ for any $h\in M$. For this it is enough to 
check that $g(\phi f^{[n]})=0$ for every $\phi\in W^{1,2}[0,1]$ and $n\in\Z_+$. First, let $n\in\N$. Then by
definition of $g_n$, we immediately have
$$
g(\phi f^{[n]})=g_n(\phi)-g_0(\rho_n\phi)=0.
$$
It remains to prove that $g(\phi f^{[0]})=0$. Using the
uniform convergence of the series $\sum\limits_{n=1}^\infty n^{-2}\rho_n$
provided by the estimate (\ref{rhon1}), we have
\begin{align*}
g(\phi f^{[0]})&=g_0(\phi)+\sum_{n=1}^\infty n^{-2}g_n(\phi)=
\int_K\Bigl(\phi'(x)+\sum_{n=1}^\infty
n^{-2}(\rho_n\phi)'(x)\Bigr)dx=
\\
&=\int_K\phi'(x)\Bigl(1+\sum_{n=1}^\infty
n^{-2}\rho_n(x)\Bigr)dx+\lim_{m\to\infty}\int_K
\phi(x)\Bigl(\sum_{n=1}^m n^{-2}\rho'_n(x)\Bigr)dx.
\end{align*}
By (\ref{rhon2}), $\sum\limits_{n=1}^\infty
n^{-2}\rho_n(x)\equiv0$. On the other hand, using (\ref{sn}) and
the equality $S_0=1$, we have
$$
\sum_{n=1}^mn^{-2}\rho'_n(x)=\sum_{n=1}^m(S_n'(x)-S_{n-1}'(x))=
S_m'(x)=A_m(x)\ \ \text{for each}\  \ x\in K.
$$
Hence
\begin{equation}
g(\phi f^{[0]})=\int_K\phi'(x)\,dx+
\lim_{m\to\infty}\int_K\phi(x)A_m(x)\,dx. \label{gn2}
\end{equation}
Integrating by parts, we obtain
$$
\int_K\phi(x)A_m(x)\,dx=-\!\int_0^1\phi(x)B'_m(x)\,dx=\!
\int_0^1\phi'(x)B_m(x)\,dx=\!\int_0^1\phi'(x)(B_m(x)+\chi(x))\,dx-\!
\int_K\phi'(x)\,dx.
$$
According to (\ref{gn2}) and the above display, 
\begin{equation}
g(\phi
f^{[0]})=\lim_{m\to\infty}\int\limits_0^1\phi'(x)(B_m(x)+\chi(x))\,dx.
\label{gn3}
\end{equation}
By (\ref{an2}) and (\ref{gn3}), $g(\phi
f^{[0]})=0$ for each $\phi\in W^{1,2}[0,1]$. Thus $g(h)=0$ for every $h\in M$ and therefore $M\neq
W^{1,2}([0,1],\ell_2)$. The proof of Theorem~\ref{ma} is complete.

\section{Remarks}

One can easily generalize Theorem~\ref{ma} by taking most any algebra of smooth functions instead of $W^{1,2}[0,1]$. For example, following the same route of argument with few appropriate amendments one can show that if $X$ is an infinite dimensional Banach space and $k\in\N$, then $C^k([0,1],X)$ as a $C^k[0,1]$-module is non-nice. We opted for $W^{1,2}([0,1],H)$ to make a point that even the friendly Hilbert space environment does not save the day. 

Theorem~\ref{ma} says that there are weird proper closed submodules of $W^{1,2}([0,1],\ell_2)$ which are not contained in any closed submodule of codimension $1$. The following question remains wide open. 

\begin{question}\label{mq1} Characterize closed submodules of $W^{1,2}([0,1],\ell_2)$.
\end{question}

%Note that applying an appropriate integral transform, one can show that the above question boils down to %characterizing the invariant subspaces of the composition operator $Tf(z)=f(\phi(z))$ acting on the Hardy space $H^2$ %of the unit disk, where $\phi(z)=\frac{i+(1-i)z}{1+i-iz}$ (a so-called parabolic non-automorphism). 

%\bigskip
%
%{\bf Acknowledgements.} \ The author is grateful to the referee for
%useful comments and suggestions, which helped to improve the
%presentation.

%\vfill\break

\small\rm

%\vskip1truecm
%
%\scshape
%
%\noindent Stanislav Shkarin
%
%\noindent Queens's University Belfast
%
%\noindent Pure Mathematics Research Centre
%
%\noindent University road, Belfast, BT7 1NN, UK
%
%\noindent E-mail address: \qquad {\tt s.shkarin@qub.ac.uk}

\end{document}